\documentclass[reqno]{amsart}
\newcommand{\sig}{\sigma}
\newcommand{\Sig}{\Sigma}
\newcommand{\ra}{\rightarrow}
\newcommand{\eps}{\epsilon}

\newcommand{\Gam}{\Gamma}

\newcommand{\lam}{\lambda}
\newcommand{\ben}{\begin{enumerate}}
\newcommand{\een}{\end{enumerate}}
\usepackage{latexsym,epsfig,amssymb,amsmath,amsthm,color,url}
\usepackage{enumitem}
\usepackage[comma,sort&compress]{natbib}
\usepackage{hyperref}

\allowdisplaybreaks 
 \numberwithin{equation}{section}
\bibfont{\footnotesize}
\newtheorem{thm}{Theorem}[section]
\newtheorem{lem}[thm]{Lemma}

\newtheorem{cor}[thm]{Corollary}
\newtheorem{notation}[thm]{Notations}
\newtheorem{example}[thm]{Example}

\theoremstyle{definition}
\newtheorem{defn}[thm]{Definition}

\theoremstyle{remark}
\newtheorem{rem}[thm]{Remark}

\title{First Order Probabilities for Galton-Watson Trees}
\date{}
\author{Moumanti Podder}
\address{Moumanti Podder, \ Courant Institute of Mathematical Sciences, \ New York University, \ 251 Mercer Street, New York, NY 10012, United States.}
\email{mp3460@nyu.edu.}
\author{Joel Spencer}
\address{Joel Spencer, \ Professor, Computer Science and Mathematics Depts, \ Courant Institute of Mathematical Sciences, \ New York University, \ 251 Mercer Street, New York, NY 10012, United States.}
\email{spencer@cims.nyu.edu.}

\begin{document}
\bibliographystyle{plainnat}

\newpage

\begin{abstract}
\emph{In the regime of Galton-Watson trees, first order logic statements are roughly equivalent to 
examining the presence of specific finite subtrees. We consider the space of all 
trees with $Poisson$ offspring distribution and show that such finite subtrees will be almost surely 
present when the tree is infinite. Introducing the notion of universal trees, we show that 
all first order sentences of quantifier depth $k$ depend only on local neighbourhoods of the root of 
sufficiently large radius depending on $k$. We compute the probabilities of these neighbourhoods 
conditioned on the tree being infinite.  We give an almost sure theory for infinite trees.}
\end{abstract}



\maketitle

\section{Introduction and main results}

For $\lam > 0$ we let $T_{\lam}$ denote the standard Galton-Watson tree, in which each node
independently has $Poisson$ offspring with mean $\lam$.  We shall set
\begin{equation}\label{defp}
   p = p(\lam) = \Pr[T_{\lam}\mbox{ is infinite}]. 
\end{equation}
As is well known, when $\lam\leq 1$, $p(\lam)=0$ while when $\lam > 1$, $p$ is
the unique positive solution to the equation
\begin{equation}\label{a30}
1-p = e^{-p\lam}.
\end{equation}

\par We let $T_{\lam}^*$ denote $T_{\lam}$ conditioned on $T_{\lam}$
being infinite.  (When using $T_{\lam}^*$ we tacitly assume $\lam > 1$.)
For any property $A$ of rooted trees we let $\Pr[A], \Pr^*[A]$
denote the probability (as a function of $\lam$) of $A$ in 
$T_{\lam},T_{\lam}^*$ respectively.  

\par The {\em first order logic} for rooted trees consists of equality ($x=y$),
parent ($\pi(x,y)$, meaning $x$ is the parent of $y$), the constant symbol
$R$ (the root), the usual Boolean connectives and existential and universal
quantification over vertices.  A {\em first order property} is a property
that can be written with a sentence $A$ in this language. The quantifier depth of any first order sentence is the number of nested quantifiers involved in expressing the sentence. We illustrate with a few examples what a typical first order sentence looks like.

\begin{example}
Consider the property that there exists a node in the tree that has precisely two children. This can be expressed in first order language as follows:
$$\exists \ u\left[\exists \ v_{1} \left[\exists \ v_{2}\left[\pi(u, v_{1}) \wedge \pi(u, v_{2}) \wedge \left[\forall \ v \left\{\pi(u, v) \implies \left\{\left(v = v_{1}\right) \vee \left(v = v_{2}\right)\right\}\right\}\right]\right]\right]\right].$$
In this particular example, the quantifier depth is $4$.
\end{example}

\begin{example}
Consider the property that the root of the tree has precisely one child and precisely one grandchild. Observe that the root of the tree being a designated symbol, this property is written in first order language as follows:
$$\exists \ u \left[\exists \ v\left[\pi(R,u) \wedge \pi(u,v) \wedge \left[\forall \ u' \left\{\pi(R,u') \implies \left(u' = u\right)\right\}\right] \wedge \left[\forall \ v' \left\{\pi(u,v') \implies \left(v' = v\right)\right\}\right]\right]\right].$$
The quantifier depth is $3$.
\end{example}

We refer the reader to \cite{MR1847951} for further discussion on first order logic.

Our main results (Theorem \ref{a52} and Corollary \ref{a54}) will be a characterization of the possible $\Pr^*[A]$,
as functions of $\lam$, where $A$ is a first order property. However, it can be shown that the property of $T$ being 
infinite is not first order.

\begin{notation}\label{notations}
Let $v\in T$, $T$ a rooted tree.
$T(v)$ denotes the subtree of $T$ that is rooted at $v$. $w$ is an $i$-descendant
of $v$ if there is a sequence $v=x_0,x_1,\ldots,x_i=w$ so that $x_j$ is the
parent of $x_{j+1}$ for $0\leq j < i$.  (We say $v$ is a $0$-descendant of
itself.) (In the Ulam-Harris notation for trees, this can be expressed as $w = (x_0,x_1,\ldots,x_i)$ where $x_{0} = v$ and $x_{i} = w$.) $w$ is a $(\leq i)$-descendant of $v$ if it is a $j$-descendant for
some $0\leq j \leq i$.
(E.g., $3$-descendants are great-grandchildren.)
We define $d(T)$ to be the depth 
of the tree, which may be infinite. For $n\geq 1$, 
$T|_{n}$ denotes the first $n$ generations of $T$, along with the root. 
That is, if $d(T) > n$, then we sever all nodes after the $n$-th generation 
(where root is the $0$-th generation) and call the truncated tree $T|_{n}$. 
If, of course, $d(T) \leq n$, then $T|_{n} = T$.  Let $T_0$ be a finite tree.
We say $T$ contains $T_0$ as a subtree if for some $v\in T$, $T(v)\cong T_0$.
We note that this is a first order property.  Letting $T_0$ have $s$ nodes, the first order 
sentence is that there exist distinct $v_1,\ldots,v_s$ having all the desired
parent relations and with $v_1,\ldots,v_s$ having no additional children.
\end{notation}

\par We use a \emph{fictitious continuation} to analyze $T_{\lam}$.  Let $X_1,X_2,\ldots$ be
a countable sequence of mutually independent and identically distributed $Poisson(\lam)$ random variables. 
Let $X_i$ be the number of children of the $i$-th node, when the tree is explored using Breadth First
Search.  (The root is considered the first node so that $X_1$ is its number of children.)  If and when 
the tree terminates (this occurs when $\sum_{i=1}^n X_i = n-1$ for the first time) the remaining (fictitious) 
$X_j$ are not used.

\begin{thm} \label{any finite tree is almost surely present}
Fix an arbitrary finite tree $T_{0}$. Consider the following statement $A$: 
\begin{equation} \label{the statement}
A := \{\text{either } T \text{ contains } T_{0} \text{ as a subtree or } T \text{ is finite}\}.
\end{equation}
Then $\Pr[A]=1$. 
\end{thm}

This is one of the main results of this paper. Note that, in particular, Theorem \ref{any finite tree is almost surely present} immediately implies that for any arbitrary but fixed finite $T_{0}$, 
\begin{equation}
\Pr^{*}[\exists \ v : T(v) \cong T_{0}] = 1.
\end{equation}
This gives us a good structural description of the infinite random Galton-Watson tree, in the sense that every local neighbourhood is almost surely present somewhere inside the tree.

\subsection{Rapidly Determined Properties}

\par We say (employing a useful notion of Donald Knuth) that an event 
is \emph{quite surely determined} in a certain parameter $s$ if the 
probability of the complement of that event is exponentially small in $s$. 
\begin{defn} \label{rapidly determined}
Consider the fictitious continuation process $T_{\lam}$. We say that an event $B$ is \emph{rapidly determined} 
if \emph{quite surely} $B$ is \emph{tautologically determined} by $X_{1}, X_{2}, \ldots, X_{s}$. Here, \emph{tautologically determined} means that for every point $\omega$ in the sample space, the realization $\left(X_{1}(\omega), X_{2}(\omega), \ldots, X_{s}(\omega)\right)$ completely determines whether the event $B$ occurs or not.
This means that for every sufficiently large $s \in \mathbb{N}$,
\begin{equation} \label{a1}
\Pr[B \text{ is not determined by } X_{1}, X_{2}, \ldots, X_{s}] \leq e^{-\beta s}
\end{equation} 
where $\beta > 0$ is independent of $s$.
\end{defn}

\begin{thm} \label{A rapidly determined}
The event $A$ described in \eqref{the statement} is a \emph{rapidly determined} property. 
\end{thm}

\par We shall now prove Theorem \ref{any finite tree is almost surely present} subject to Theorem \ref{A rapidly determined}. Fix an arbitrary finite $T_{0}$. Assume Theorem \ref{any finite tree is almost surely present} is false so that $\Pr[A] < 1$, where $A$ is as in \eqref{the statement}.
For each $s\in \mathbb{N}$, with probability at least $1-\Pr[A]$ the values $X_1,\ldots,X_s$ do
not terminate the tree, nor do they force a copy of $T_{0}$.  Then $A$ would not be tautologically
determined.  So $A$ would not be \emph{rapidly determined} and Theorem \ref{A rapidly determined}
would be false.  Taking the contrapositive, Theorem \ref{A rapidly determined} implies
Theorem \ref{any finite tree is almost surely present}.  We prove Theorem \ref{A rapidly determined}
in \S \ref{sectrapdet}. 

\begin{rem} \label{conclusion of first theorem}
The conclusion of Theorem \ref{any finite tree is almost surely present} is 
really that, fixing any finite tree $T_{0}$, $T_{\lam}^*$
contains $T_0$ as a subtree with probability one.
We can say a bit more.  Let $T_0$ have root $v$.  For $L\geq 1$ define
$T_0[L]$ by adding $L$ new points $v_0,\ldots,v_{L-1}$ and making $v_i$ a child of $v_{i-1}$,
$1\leq i\leq L-1$ and $v$ a child of $v_{L-1}$.  $T_{\lam}^*$ contains $T_0[L]$ with
probability one.  But then it contains a $T_0$ where the root of $T_0$ is at least distance
$L$ from the root of $T$.  We thus deduce that for any finite $T_0$ and any $L$ there
will, with probability one in $T_{\lam}^*$, be a $v$ at distance at least $L$ from the
root such that $T(v)\cong T_0$.
\end{rem}

\subsection{Ehrenfeucht Games}\label{sectehr}

\par We use a very standard and well-known tool to analyze first order properties on rooted trees, namely the Ehrenfeucht games. The Ehrenfeucht games are what bridges the gap between mathematical logic and a complete structural description of logical statements on graphs. Fix a positive integer $k$. The standard $k$-move Ehrenfeucht game used to analyze first order properties partitions the space of all rooted trees into finitely many equivalence classes. Any two trees belonging to the same equivalence class if and only if they have the same truth value for every first order property of quantifier depth $\leq k$. That is, given a first order sentence $A$ of quantifier depth at most $k$, if $A$ holds true for one of the trees in an equivalence class, then it holds true for all others in that class as well. This notion is made more precise in the following exposition.
\par We begin with describing the {\em standard} game, and later move on to a more specialized variant of the game that is suited to our analysis. Fix $k\geq 1$ and two trees $T_{1}$ rooted at $R_{1}$ and $T_{2}$ rooted at $R_{2}$ (these are known to both players). The Ehrenfeucht game $EHR[T_1,T_2;k]$ is a $k$-round game between two players, the Spoiler and the Duplicator. In each round Spoiler picks a vertex from \emph{either} $T_1$ or $T_2$ and then Duplicator picks a vertex from the other tree.  Letting $x_1,\ldots,x_k$; $y_1,\ldots,y_k$ be the vertices selected (in that order) from $T_1,T_2$ respectively, Duplicator wins if \emph{all} of the following hold:
\ben 
\item $x_i=R_1$ iff $y_i=R_2$;
\item $\pi(x_i,x_j)$ iff $\pi(y_i,y_j)$, i.e. $x_{i}$ is the parent of $x_{j}$ if and only if $y_{i}$ is the parent of $y_{j}$;
\item $\pi(R_1,x_i)$ iff $\pi(R_2,y_i)$, i.e., if $x_{i}$ is a child of the root $R_{1}$, then $y_{i}$ is a child of $R_{2}$, and vice versa;
\item $x_i=x_j$ iff $y_i=y_j$.
\een
We write $T_1\equiv_k T_2$ if and only if Duplicator
wins $EHR[T_1,T_2;k]$. This equivalence relation partitions all rooted trees into finitely many equivalence classes. It can be shown that two rooted trees $T_1,T_2$ (with roots $R_1,R_2$) have the same $k$-Ehrenfeucht value iff they satisfy precisely the same first order properties of quantifier depth at most $k$.


\par We shall now describe the promised modified version of the game.
Let $T$ be a rooted tree, $v\in T$, and $r>0$.  Let $T^-$ be the (undirected)
tree on the same vertex set with $x,y$ adjacent iff one of them is the parent of
the other.  Let $B_{T}(v;r)$ denote the ball of radius $r$ around $v$.  That is,
\begin{equation}\label{ball} B_{T}(v; r) = \{u \in T: d(u, v) < r \text{ in } T^- \} 
\end{equation}
Here $d(\cdot, \cdot)$ gives the usual graph distance. (For example, cousins are
at distance four.)
\par Let $k$ (the number of rounds) and $M$ (an upper bound on the maximal distance) be fixed.
Let $T_1,T_2$ be trees with \emph{designated nodes}
$v_{1} \in T_{1}, v_{2} \in T_{2}$. Set
$$B_{i} = B_{T_{i}}(v_{i}; \lfloor M/2 \rfloor), \quad i = 1, 2.$$ 
The $k$-move $M$-distance preserving Ehrenfeucht game, 
denoted by $EHR_{M}[B_{1}, B_{2}; k]$, is played 
on these balls. We add a round zero in which the moves $v_1,v_2$ must be played.
(Essentially these are designated vertices.)  As before,
each round ($1$ through $k$) Spoiler picks a vertex from either $T_1$ or
$T_2$ and then Duplicator picks a vertex from the other tree.  Letting
$x_0,\ldots,x_k$; $y_0,\ldots,y_k$ be the vertices selected from
$T_1,T_2$ respectively, Duplicator wins if
\begin{itemize}
\item For $0\leq i, j \leq k$, $d(x_i,x_j)=d(y_i,y_j)$.  Equivalently, for all $1\leq s \leq M$ and
all $0\leq i, j \leq k$, $d(x_i,x_j)=s$ if and only if $d(y_i,y_j)=s$.
\item For $0 \leq i, j \leq k$, $\pi(x_i,x_j)$ iff $\pi(y_i,y_j)$.
\item For $0\leq i, j \leq k$, $x_i=x_j$ iff $y_i=y_j$.
\end{itemize}

\par Two balls $B_{1}, B_{2}$ (as described above) are said to have \emph{the same $(M; k)$-Ehrenfeucht 
value} if Duplicator wins $EHR_{M}[B_{1}, B_{2}; k]$. We denote this by 
\begin{equation}\label{Mkequiv}
B_{1} \equiv_{M; k} B_{2}
\end{equation}
This being an equivalence relation, the space of all such balls with designated centers, 
is partitioned into \emph{$(M; k)$-equivalence classes}. We let $\Sigma_{M; k}$ denote the set of all 
$(M; k)$-equivalence classes. 

\par We create a first order language consisting of $=,\pi(x,y)$ and $d(x,y)=s$ for $1\leq s\leq M$ (note that $s$ is \emph{not} a variable here).
There are only finitely many binary predicates (relations involving two variables).  (In general adding the distance function 
would add an unbounded number of binary predicates.  In our case, however, the diameter is
bounded by $M$ and so we are only adding the $M$ predicates $d(x,y)=s, 1 \leq s \leq M$.)
Hence the number of equivalence classes corresponding to this game will also be finite. 
That is,  $\Sigma_{M; k}$ is a finite set.\\

\subsection{Universal Trees}\label{sectuniv}

\par A \emph{universal tree}, as defined below, shall have the property that once
$T$ contains it, all first order statements up to quantifier depth $k$ depend
only on the local neighborhood of the root.

\begin{defn} \label{universal tree}
Fix a positive integer $k$. Let 
\begin{equation} \label{M value}
M_{0} = 2 \cdot 3^{k+1}.
\end{equation} 
A finite tree $T_{0}$ will be called \emph{universal} if the following phenomenon happens: 
Take any two trees $T_{1}, T_{2}$ with roots $R_{1}, R_{2}$ such that:
\begin{enumerate}
\item the $3^{k+1}$ neighbourhoods around the root have the same $(M_{0}; k)$ value, i.e. 
\begin{equation} \label{neighbourhood around the roots}
B_{T_{1}}(R_{1}; 3^{k+1}) \equiv_{M_{0}; k} B_{T_{2}}(R_{2}; 3^{k+1}).
\end{equation}
\item For some $u_{1} \in T_{1}, u_{2} \in T_{2}$ such that
\begin{equation} \label{distance of universal tree from root more}
d(R_{1}, u_{1}) > 3^{k+2}, \quad d(R_{2}, u_{2}) > 3^{k+2},
\end{equation}
we have 
\begin{equation}\label{a97}
T_1(u_1) \cong T_2(u_2) \cong T_0.
\end{equation}
\end{enumerate}
Then $T_1\equiv_k T_2$.  Equivalently, Duplicator wins the \emph{standard} $k$-move Ehrenfeucht 
game played on $T_{1}, T_{2}$. 
\end{defn}

\begin{rem}
Technically, we should call such a $T_{0}$ as described in Definition \ref{universal tree} $k$-\emph{universal}. However, in the sequel, we simply refer to this as \emph{universal} for the convenience of notation, and since the dependence on $k$ will be clear in each context.
\end{rem}

\par We prove in Theorem \ref{xmastree} that such a \emph{universal tree} indeed exists, by imposing sufficient structural conditions on it.

\begin{rem} \label{how neighbourhood of root suffices}
Fix a certain \emph{universal tree} $UNIV_{k}$, given $k \in \mathbb{N}$. Using theorem \ref{any finite tree is almost surely present}, 
we conclude that $T_{\lam}^{*}$ will almost surely contain $UNIV_{k}$. From Remark \ref{conclusion of first theorem}, we 
say further that there will almost surely exist a node $v$ at distance $> 3^{k+2}$ from the root such 
that $$T(v) \cong UNIV_{k}.$$ From the definition of \emph{universal trees}, then the \emph{standard} 
Ehrenfeucht value of $T_{\lam}^{*}$ will be determined by the $(M_{0}; k)$-Ehrenfeucht value of $B_{T_{\lam}^{*}}(R; 3^{k+1})$, 
the $3^{k+1}$-neighbourhood of the root $R$, where $M_{0}$ is as in \eqref{M value}.
\end{rem}

\subsection{An Almost Sure Theory}\label{a.s.theory:subsec}

\par Let $\mathcal{B}_{i}, 1 \leq i \leq N$ for some positive integer $N$, denote the finitely many $(M_{0}; k)$-equivalence 
classes. Note that these are defined on balls of radius $3^{k+1}$ centered at a designated vertex which is a node in some tree. Then for every realization $T$ of $T_{\lam}^{*}$, 
\begin{equation}\label{justone}
B_{T}(R; 3^{k+1}) \in \mathcal{B}_{i} \quad \text{for precisely one }i, 1 \leq i \leq N.
\end{equation} 
Almost surely for two realizations $T_{1}, T_{2}$ of $T_{\lambda}^{*}$ which have the same local neighbourhoods of the roots, i.e.
$$B_{T_{1}}(R_{1}; 3^{k+1}) \in \mathcal{B}_{i}, \quad B_{T_{2}}(R_{2}; 3^{k+1}) \in \mathcal{B}_{i} \quad \text{for the same } i,$$ 
we have $T_{1} \equiv_{k} T_{2}$.   As the $\mathcal{B}_i$ are equivalence classes over the space of rooted trees they may be considered
properties of rooted trees and so have probabilities $\Pr^{*}[\mathcal{B}_i]$ in $T_{\lam}^{*}$.  As they finitely partition the
space of all rooted trees
\begin{equation}\label{a6}
\sum_{i=1}^{N} \Pr^{*}[\mathcal{B}_{i}] = 1.
\end{equation}

\par Let $\mathcal{AS}$ denote the almost sure theory for $T_{\lam}^*$. That is, $\mathcal{AS}$ consists of all first order sentences $B$ such that
$\Pr^*[B]=1$.  We now give an axiomatization of $\mathcal{AS}$. Let $\mathcal{T}$ be defined by the following schema:
\begin{equation}\label{Tschema} 
A\left[T_{0}\right] := \left\{\exists \ v: T(v) \cong T_{0}\right\}, \ \text{for all } T_{0} \text{ finite trees}.
\end{equation}

\begin{thm} \label{nearly complete theory}
Under the probability $\Pr^{*}$,
\begin{equation}\label{a9}
\mathcal{T} = \mathcal{AS}
\end{equation}
That is, the first order statements $B$ with $\Pr^{*}[B] = 1$ are precisely those statements derivable from the axiom schema $\mathcal{T}$.
\end{thm}

As $\mathcal{T}$ does not depend on $\lam$ we also have:

\begin{cor}\label{allsame}  The almost sure theory $\mathcal{AS}$ is the same for all $\lam > 1$.
\end{cor}

That $\mathcal{T} \subseteq \mathcal{AS}$ is already clear from Theorem \ref{any finite tree  is almost surely present}. 
To show the reverse inclusion, consider for every $1 \leq i \leq N$, $\mathcal{T} + \mathcal{B}_{i}$.  In this theory
every finite $T_0$ is contained as a subtree and 
the $3^{k+1}$-neighbourhood of the root belongs to the equivalence class $\mathcal{B}_{i}$.
As discussed above in Remark \ref{how neighbourhood of root suffices}, this set of information completely determines the \emph{standard} 
Ehrenfeucht value of the infinite tree. That is, for any first order sentence $A$ of quantifier depth $k$ 
\begin{equation}\label{eitheror}
\text{either } \mathcal{T} + \mathcal{B}_{i} \models A \quad \text{or } \mathcal{T} + \mathcal{B}_{i} \models \neg A.
\end{equation}
The standard notation $T \models A$ for a tree $T$ and a property $A$ means that the property $A$ holds true for tree $T$. Set
\begin{equation}\label{JA} J_{A} = \{1 \leq i \leq N: \mathcal{T} + \mathcal{B}_{i} \models A\}. \end{equation}
Under $T_{\lam}^*$, $A$ holds \emph{if and only if} $\mathcal{B}_i$ holds for some $i\in J_A$.  Thus we can express
\begin{equation}\label{anyA} \Pr^*[A] = \sum_{i\in J_A} \Pr^*[\mathcal{B}_i].  \end{equation}
In Section \ref{sectprob}  we shall use this to express all $\Pr^*[A]$ in reasonably succint form.

Now suppose, under $T_{\lam}^*$, that $\Pr^{*}[A] = 1$.  As the $\mathcal{B}_i$ partition the neighbourhoods around the roots of trees, this
implies  that $J_{A} = \{1, 2, \ldots N\}$. That is, $\mathcal{T} + \mathcal{B}_{i} \models A$ 
for all $1 \leq i \leq N$ and $\bigvee_{i=1}^N \mathcal{B}_i$ is a tautology. 
Hence $A$ is derivable from $\mathcal{T}$. Thus $\mathcal{AS} \subseteq \mathcal{T}$.\\

\par In Section \ref{sectprob} below, we turn to the computation of the possible $\Pr^*[A]$.
As seen above, in the 
space of $T_{\lam}^{*}$, the neighbourhoods around the root of sufficiently large radius are instrumental in 
determining the \emph{standard} 
Ehrenfeucht value of the tree. It only makes sense, therefore, to compute the probabilities of having specific 
neighbourhoods around the 
root conditioned on the tree being infinite. We shall do this in a recursive 
fashion, using induction on the number of generations below the root that we are considering.\\

\section{Containing All Finite Trees}\label{sectinf}

\subsection{A Rapidly Determined Property}\label{sectrapdet}

\par We prove here Theorem \ref{A rapidly determined}. We fix an arbitrary finite tree $T_{0}$ with depth $d(T_{0}) = d_{0}$, following the notation given in \ref{notations}. We alter the \emph{fictitious continuation} process $T_{\lam}$ described previously. If for some finite, first $n \in \mathbb{N}$, we have
$\sum_{i=1}^n X_i = n-1$, then the actual tree has vertices $1,\ldots,n$. If this phenomenon does not happen for any finite $n$, then we have one infinite tree described by our fictitious continuation. If the tree does abort after $n$ vertices, we begin a new tree with vertex $n+1$ as the root, and generate it from $X_{n+1},X_{n+2},\ldots$.  Again, if this tree terminates at some $n_1$ we begin a new tree with vertex $n_1+1$. Continuing, we generate an infinite forest, with vertices the positive integers. We call this the {\em forest} process and label it $T_{\lam}^{for}$.

Then we define, for every $s \in \mathbb{N}$, the event (in $T_{\lam})$
\begin{equation} \label{good event}
good(s) = \{A \text{ is completely determined by } X_{1}, \ldots X_{s}\},  
\end{equation}
where $A$ is as in \eqref{the statement}. Set $bad(s) = good(s)^{c}$. 
For every node $i \in \mathbb{N}$, define in $T_{\lam}^{for}$ 
\begin{equation}\label{Isubi}
I_{i} = \mathbf{1}_{T(i) \cong T_{0}}.
\end{equation}
That is, $I_i$ is the indicator function of the event that {\em in the random forest}
$T(i)\cong T_0$.  Set
\begin{equation}
Y = \sum_{i=1}^{\lfloor \epsilon^{d_{0}} s \rfloor} I_{i},
\end{equation}
where, with foresight, we require 
\begin{equation} \label{condition on epsilon}
0 < \epsilon < \frac{1}{\lam + 1}.
\end{equation}
(Our $\eps$ is chosen sufficiently small so that quite surely in $s$,
in $T_{\lam}^{for}$, 
all of the $(\leq d_0)$-descendants $j$ of all $i\leq s\eps$ have
$j\leq s$.)  
We create a martingale, setting, for $1 \leq i \leq s$,
\begin{equation} \label{the martingale}
Y_{i} = E[Y| X_{1}, X_{2}, \ldots X_{i}], \quad Y_{0} = E[Y].
\end{equation}

\par 
In $T^{for}_{\lam}$, 
for $x \in \mathbb{R}^{+}, i \in \mathbb{N}$, 
set 
\begin{equation}\label{a22}
\mathcal{S}_{i}(x) = \{\text{indices of all } i\text{-descendants of nodes } 1, 2, \ldots \lfloor x \rfloor\}
\end{equation} 
with $\mathcal{S}_{0}(x) = \{1, 2, \ldots \lfloor x \rfloor\}$, where an $i$-descendant is as described in Notations \ref{notations}. Define, for $i \in \mathbb{N}$,
\begin{equation}\label{a23}
g_{i}(x) = \text{highest index 
recorded in } \displaystyle \bigcup_{j=0}^{i} \mathcal{S}_{j}(x).
\end{equation}

\begin{lem} \label{highest index of a d-child}
For any $x \in \mathbb{R}^{+}, d \in \mathbb{N}$, 
\begin{equation}\label{a33}
g_{d}(x) = g_{1}^{d}(x). 
\end{equation}
Here $g_1^d$ denotes the $d$-times composition of $g_1$.
\end{lem}

\begin{proof}
We prove this using induction on $d$. For $d = 1$ this is true by definition of $g_{1}$. For $d = 2$, the highest possible 
index of all the children and grandchildren of $1, 2, \ldots \lfloor x \rfloor$ is equal to the highest index of the 
children of the nodes $1, 2, \ldots g_{1}(\lfloor x \rfloor) = g_{1}(x)$, which is $g_{1}(g_{1}(x))$. Now 
suppose we have proved the claim for some $d \in \mathbb{N}, d \geq 2$. Once again, a similar argument comes into 
play. The highest index among all the $(d+1)$-descendants of nodes $1, 2, \ldots \lfloor x \rfloor$, is also equal to the highest 
index among all the $d$-descendants of the nodes $1, 2, \ldots g_{1}(x)$, which by induction hypothesis 
will be $g_{1}^{d}(g_{1}(x)) = g_{1}^{d+1}(x)$.
\end{proof}


\par When $g_{d_{0}}(\lfloor \epsilon^{d_{0}} s \rfloor) \leq s$, the descendents $j$ of
$1,\ldots, \lfloor \epsilon^{d_{0}} s \rfloor$ down to generation $d_0$ will all have $j\leq s$.  Thus
$Y$ will be completely determined by $X_{1}, \ldots X_{s}$. That is,
\begin{equation} \label{sufficient condition}
g_{1}^{d_{0}}(\lfloor \epsilon^{d_{0}} s \rfloor) \leq s \quad \Rightarrow \quad Y_{s} = Y.
\end{equation}

\par A few observations about the function $g_{1}(\cdot)$ are important. First, 
\begin{equation}\label{a24}
g_{1}(x) \geq \lfloor x \rfloor \quad \text{for all } x \in \mathbb{R}^{+}.
\end{equation}  
In $T_{\lam}^{for}$
every time the tree terminates, we start a new tree, and that uses up one extra index for the root of the new tree. 
But while 
counting the nodes $1, 2, \ldots \lfloor x \rfloor $, for any $x \in \mathbb{R}^{+}$, at most $\lfloor x \rfloor$ many new trees 
need be started. Therefore
\begin{equation} \label{first upper bound for g1}
g_{1}(x) \leq \lfloor x \rfloor + \sum_{i=1}^{\lfloor x \rfloor} X_{i}.
\end{equation}
Further, by the definition of $g_{1}(\cdot)$, it is clear that it is monotonically increasing. 

\par We shall use the inequality in \eqref{first upper bound for g1} to show that, for $\epsilon$ as chosen in \eqref{condition on epsilon}, quite surely in $s$, we have $Y_{s} = Y$, i.e. $Y$ is tautologically determined by $X_{1}, \ldots, X_{s}$ with exponentially small failure probability in $s$. This involves showing that for $i$ this small, i.e. $1 \leq i \leq \left\lfloor \epsilon^{d_{0}} s\right\rfloor$, $T(i)$ is quite surely determined by $X_{1}, \ldots, X_{s}$.

\par We employ Chernoff bounds.
For $x \in \mathbb{R}^{+}$ and any $\alpha > 0$,
\begin{align} \label{upper bound for g1}
\Pr[g_{1}(\epsilon x) > x] =& \Pr[e^{\alpha g_{1}(\epsilon x)} > e^{\alpha x}] \nonumber\\
\leq & E[e^{\alpha g_{1}(\epsilon x)}] e^{-\alpha x} \nonumber\\
\leq & E[e^{\alpha (\epsilon x + \sum_{i=1}^{\lfloor \epsilon x \rfloor} X_{i})}] e^{-\alpha x} \nonumber\\
=& e^{\alpha \epsilon x} \prod_{i=1}^{\lfloor \epsilon x \rfloor} E[e^{\alpha X_{i}}] e^{-\alpha x} \nonumber\\
=& e^{-(1-\epsilon)\alpha x} \left\{\exp \left[\lam \left(e^{\alpha} - 1\right)\right]\right\}^{\lfloor \epsilon x \rfloor} \nonumber\\
\leq & e^{-(1-\epsilon)\alpha x} \left\{\exp \left[\lam \left(e^{\alpha} - 1\right)\right]\right\}^{\epsilon x} \nonumber\\
=& \exp\{-[(1 - \epsilon)\alpha - \lam (e^{\alpha} - 1)\epsilon]x\}. 
\end{align}

\par It can be checked that for any $\alpha \in \left(0, \log \left(\frac{1-\epsilon}{\lam\epsilon}\right)\right)$, the 
exponent in \eqref{upper bound for g1} is negative, i.e. $-[(1 - \epsilon)\alpha - \lam (e^{\alpha} - 1)\epsilon] < 0$. 
We set
\begin{equation}\label{a25}
\eta = [(1 - \epsilon)\alpha - \lam (e^{\alpha} - 1)\epsilon]
\end{equation}  
Observe that $\eta$ is positive. Now we have the upper bound:
\begin{equation} \label{concise upper bound for g1}
\Pr[g_{1}(\epsilon x) > x] \leq e^{-\eta x}.
\end{equation}

\par We make the following claim:

\begin{lem} \label{upper bound for general d}
For any $d, s \in \mathbb{N}$, 
\begin{equation}
\Pr[g_{1}^{d}(\epsilon^{d} s) > s] \leq \sum_{i=0}^{d-1}e^{-\epsilon^{i} \eta s}.
\end{equation}
\end{lem}

\begin{proof}
We prove this using induction on $d$. We have already seen that this holds for $d = 1$. This initiates the induction hypothesis. Suppose it holds for some $d \in \mathbb{N}$. Then
\begin{align}
\Pr[g_{1}^{d+1}(\epsilon^{d+1} s) > s] =& \Pr[g_{1}^{d+1}(\epsilon^{d+1} s) > s, g_{1}(\epsilon^{d+1} s) > \epsilon^{d} s] + \Pr[g_{1}^{d+1}(\epsilon^{d+1} s) > s, g_{1}(\epsilon^{d+1} s) \leq \epsilon^{d} s] \nonumber\\
\leq & \Pr[g_{1}(\epsilon^{d+1} s) > \epsilon^{d} s] + \Pr[g_{1}^{d}(\epsilon^{d} s) > s] \nonumber\\
\leq & e^{-\eta \cdot \epsilon^{d} s} + \sum_{i=0}^{d-1}e^{-\epsilon^{i} \eta s}, \quad \text{by induction hypothesis and } \eqref{concise upper bound for g1}; \nonumber\\
=& \sum_{i=0}^{d} e^{-\epsilon^{i} \eta s}. \nonumber
\end{align}
This completes the proof.
\end{proof}

\par From Lemma \ref{upper bound for general d}, we conclude that 
\begin{equation}\label{a26}
\Pr[g_{1}^{d_{0}}(\lfloor \epsilon^{d_{0}} s \rfloor) > s] \leq \sum_{i=0}^{d_{0}-1}e^{-\epsilon^{i} \eta s},
\end{equation} From \eqref{sufficient condition}, this means
\begin{equation} \label{globalbad upperbound}
\Pr[Y_{s} = Y] \geq 1 - \sum_{i=0}^{d_{0}-1}e^{-\epsilon^{i} \eta s}.
\end{equation}
As promised earlier, we therefore have that, quite surely, $Y_{s} = Y$. In the following definition, we describe the event $Y_{s} = Y$ as $globalgood(s)$, emphasizing the dependence on the parameter $s$. What we can conclude from the above computation is that $globalgood(s)$ fails to happen with only exponentially small failure probability in $s$.

\begin{defn}\label{globalgooddef}  $globalgood(s)$ is the event $Y_s=Y$.  $globalbad(s)$ is
the complement of $globalgood(s)$.
\end{defn}

\par  We now claim that the 
martingale $\{Y_{i}: 0 \leq i \leq s\}$
satisfies a Lipschitz Condition.

\begin{lem}\label{lipschitz}
There exists constant $C > 0$ such that for $1 \leq i \leq s$, 
\begin{equation}
|Y_{i} - Y_{i-1}| \leq C.
\end{equation}
\end{lem}

\begin{proof}
For $1 \leq i \leq \lfloor \epsilon^{d_{0}} s \rfloor$, fix a sequence $\vec{x} = (x_{1}, \ldots x_{i-1}) \in (\mathbb{N} \cup \{0\})^{i-1},$ and then 
consider $$y_{i} = E[Y|X_{1} = x_{1}, \ldots X_{i-1} = x_{i-1}, X_{i}] = \sum_{1 \leq j \leq \lfloor \epsilon^{d_{0}} s \rfloor}E[I_{j}| X_{1} = x_{1}, \ldots X_{i-1} = x_{i-1}, X_{i}]$$ 
and
$$y_{i-1} = E[Y|X_{1} = x_{1}, \ldots X_{i-1} = x_{i-1}] = \sum_{1 \leq j \leq \lfloor \epsilon^{d_{0}} s \rfloor} E[I_{j} | X_{1} = x_{1}, \ldots X_{i-1} = x_{i-1}].$$

\par $I_{j}$ will be affected by the extra information about $X_{i}$ only if either $j=i$ or node $j$ is an ancestor 
of node $i$ at distance $\leq d_{0}$ from $i$. If $j = i$, then it will of course affect the conditional expectation 
because $X_{i}$ gives the number of children of $j$ in that case. When $j > i$, this is immediate, because any subtree 
rooted at $j$ has no involvement of $X_{i}$. When $j < i$, but not an ancestor of $i$, $i$ is not a part of the 
subtree $T(j)$ rooted at $j$. Therefore $X_{i}$, the number of children of node $i$, does not contribute 
anything to the probability of the presence of $T_{0}$ rooted at $j$. When $j$ is an ancestor of $i$ 
but at distance $> d_{0}$ from $i$, $i$ won't be a part of the subtree $T(j)|_{d_{0}}$ at all.

\par When $j$ is an ancestor of $i$ and at 
distance $d_{0}$ from $i$, then $i$ is a leaf node of $T(j)|_{d_{0}}$ and therefore $X_{i}$, the number of 
children of $i$, will actually play a role, because to ensure that $T(j) \cong T_{0}$, the leaf nodes of $T(j)|_{d_{0}}$ must have no children of their own in $T_{\lam}^{for}$.

\par That is, we need be concerned with the at most $d_{0}$ ancestors of node $i$, plus $i$ iteself, and 
for each of them, the 
difference in the conditional expectations of $I_{j}$ can be at most $1$. Denoting by $\sum^{*}$ the sum 
over $j = i$ and $j$ an ancestor of $i$ at distance $\leq d_{0}$ from $i$, this gives us:
\begin{align}
|y_{i} - y_{i-1}| =& \left| \sum^{*} E[I_{j}|X_{1} = x_{1}, \ldots X_{i-1} = x_{i-1}, X_{i}] - E[I_{j}|X_{1} = x_{1}, \ldots X_{i-1} = x_{i-1}] \right| \nonumber\\
\leq & \sum^{*}\left| E[I_{j}|X_{1} = x_{1}, \ldots X_{i-1} = x_{i-1}, X_{i}] - E[I_{j}|X_{1} = x_{1}, \ldots X_{i-1} = x_{i-1}] \right| \nonumber\\
\leq & d_{0}+1. \nonumber
\end{align}
The final inequality follows from the argument above that $\sum^{*}$ involves summing over at most $d_{0}+1$ many terms, and each summand is at most $1$, since each summand is the difference of the expectations of indicator random variables. This proves Lemma \ref{lipschitz}, with $C = d_{0}+1$.

\end{proof}

\par Given Lemma \ref{lipschitz} we  apply  Azuma's inequality. Consider the martingale
$$Y'_{i} = \frac{E[Y] - Y_{i}}{d_{0}+1}, \quad 0 \leq i \leq s.$$
Set, for a typical node $v$ in a random Galton-Watson tree $T$ with $Poisson(\lam)$ offspring distribution, 
\begin{equation}\label{setp0}
\Pr[T(v) \cong T_{0}] = p_{0},
\end{equation}  so that 
$E[Y] = \lfloor \epsilon^{d_{0}}  s\rfloor p_{0}$. Applying Azuma's inequality to $\{Y'_{i}, 0 \leq i \leq s\}$, for any $\beta > 0$,
$$\Pr[Y'_{s} > \beta \sqrt{s}] < e^{-\beta^{2}}.$$ We choose $$\beta = \frac{\epsilon^{d_{0}} p_{0} \sqrt{s}}{2 (d_{0}+1)}.$$ This gives
\begin{equation}
\Pr[Y_{s} < \frac{\epsilon^{d_{0}} p_{0}}{2} \cdot s - p_{0}] < \exp\left\{-\frac{\epsilon^{2d_{0}} p_{0}^{2}}{4 (d_{0}+1)^{2}} \cdot s \right\}.
\end{equation}
Writing $$\xi = \frac{\epsilon^{d_{0}} p_{0}}{2}, \quad \varphi = \frac{\epsilon^{2(d_{0})} p_{0}^{2}}{4 (d_{0}+1)^{2}},$$
we can rewrite the above inequality as
\begin{equation} \label{upper bound for Ys}
\Pr[Y_{s} < \xi s - p_{0}] < e^{-\varphi s}.
\end{equation}

\par Putting everything together, we get for all $s$ large enough:
\begin{align}
\Pr[Y = 0] =& \Pr[Y = 0, Y_{s} = Y] + \Pr[Y = 0, Y_{s} \neq Y] \nonumber\\
\leq & \Pr[Y_{s} < \xi s - p_{0}] + \Pr[Y_{s} \neq Y] \nonumber\\
\leq & e^{-\varphi s} + \sum_{i=0}^{d_{0}-2}e^{-\epsilon^{i} \eta s}; \quad \text{from } \eqref{globalbad upperbound} \text{ and } \eqref{upper bound for Ys}; \nonumber
\end{align}
which is an upper bound exponentially small in $s$. This gives us the proof of Theorem \ref{A rapidly determined}.

\section{Universal trees exist!} \label{sectionuniv}
\par In this section, we shall establish sufficient conditions that guarantee the existence of \emph{universal trees}. 
Fixing $k \in \mathbb{N}$, set $M_{0} = 2 \cdot 3^{k+1}$ as in \eqref{M value}. {\em Assume} $T_{0}$ is a finite tree 
with root $R_{0}$ with the following properties:
\begin{enumerate}
\item \label{point 1} For every $\sigma \in \Sigma_{M_{0}; k}$, there are distinct nodes 
$v_{i; \sigma} \in T_{0}, 1 \leq i \leq k$, 
with the following conditions satisifed: for every $\sigma \in \Sigma_{M_{0}; k}$ and every $1 \leq i \leq k$, we have
\begin{equation} \label{distance from root of universal tree}
d(R_{0}, v_{i; \sigma}) > 3^{k+2};
\end{equation}
for every $\sigma_{1}, \sigma_{2} \in \Sigma_{M_{0}; k}$ and $1 \leq i_{1}, i_{2} \leq k$, with $(\sigma_{1}, i_{1}) \neq (\sigma_{2}, i_{2})$, we have
\begin{equation} \label{distance from each other}
d(v_{i_{1}; \sigma_{1}}, v_{i_{2}; \sigma_{2}}) > 3^{k+2};
\end{equation}
and for all $1 \leq i \leq k, \sigma \in \Sigma_{M_{0}; k}$,
\begin{equation} \label{each equivalence class represented}
B(v_{i; \sigma}; 3^{k+1}) \in \sigma.
\end{equation}

\item \label{point 2}For every $1 \leq i \leq k$, every choice of $u_{1}, \ldots u_{i-1} \in T_{0}$, and every choice of $\sigma \in \Sigma_{M_{0};k}$, there exists a vertex $u_{i} \in T_{0}$ such that 
\begin{equation} \label{distance from previous vertices}
d\left(u_{i}, u_{j}\right) > 3^{k+2}, \text{ for all } 1 \leq j \leq i-1,
\end{equation}
\begin{equation} \label{distance from root}
d\left(R_{0}, u_{i}\right) > 3^{k+2},
\end{equation}
and 
\begin{equation}\label{copy of each equivalence class}
B\left(u_{i}; 3^{k+1}\right) \in \sigma.
\end{equation}

\end{enumerate}

\begin{rem}
Observe that Condition (\ref{point 2}) is stronger than Condition (\ref{point 1}) and actually implies the latter. However, for pedagogical clarity, and since (\ref{point 1}) gives a nice structural description of the \emph{Christmas tree} that is described in Theorem \ref{xmastree}, we retain (\ref{point 1}). Furthermore, we state (\ref{point 1}) before (\ref{point 2}) since, we feel, it is an easier condition to visualize.
\end{rem}

\begin{lem} \label{sufficient condition for universal tree}
$T_{0}$ with properties described above will be a universal tree.
\end{lem}

\begin{proof}
Recall the definition of universal trees. We start with two trees $T_{1}, T_{2}$ with roots $R_{1}, R_{2}$, and which satisfy the following conditions:
\begin{enumerate}
\item The balls $B(R_{1}; 3^{k+1}), B(R_{2}; 3^{k+1})$ satisfy
\begin{equation} \label{around the root}
B(R_{1}; 3^{k+1}) \equiv_{M_{0}; k} B(R_{2}; 3^{k+1}).
\end{equation}
\item For some $u_{1} \in T_{1}, u_{2} \in T_{2}$ such that
\begin{equation} \label{distance of universal tree from root}
d(R_{1}, u_{1}) > 3^{k+2}, \quad d(R_{2}, u_{2}) > 3^{k+2},
\end{equation}
we have each of $T_{1}(u_{1})$ and $T_{2}(u_{2})$ isomorphic to $T_{0}$. If $\varphi_{1} : T_{0} \rightarrow T_{1}(u_{1}), \varphi_{2}: T_{0} \rightarrow T_{2}(u_{2})$ are these isomorphisms, then $$\varphi_{1}(v_{i; \sigma}) = v_{i; \sigma}^{(1)}, \quad \varphi_{2}(v_{i; \sigma}) = v_{i; \sigma}^{(2)},$$ for all $\sigma \in \Sigma_{M_{0}; k}, 1 \leq i \leq k$.

\end{enumerate}

\par Now we give a winning strategy for the Duplicator. We assume that since $R_{1}, R_{2}$ are designated vertices, $x_{0} = R_{1}, y_{0} = R_{2}$. Let $(x_{i}, y_{i})$ be the pair chosen from $T_{1} \times T_{2}$ in the $i$-th move, for $1 \leq i \leq k$. Now, we claim the following:\\

\par The Duplicator can play the game such that, for each $0 \leq i \leq k$,
\begin{itemize}
\item he can maintain $$B(x_{i}; 3^{k+1-i}) \equiv_{M_{0}; k} B(y_{i}; 3^{k+1-i}),$$
(Our proof only needs $$B(x_{i}; 3^{k+1-i}) \equiv_{M_{0}; k-i} B(y_{i}; 3^{k+1-i}),$$ but the stronger assumption is a bit
more convenient);
\item for all $0 \leq j < i$ such that $x_{j} \in B(x_{i}; 3^{k+1-i})$, the corresponding $y_{j} \in B(y_{i}; 3^{k+1-i})$, and vice versa, according to the winning strategy of $EHR_{M_{0}}[B(x_{i}; 3^{k+1-i}), B(y_{i}; 3^{k+1-i}); k]$. Again, this is overkill as one need only consider the
Ehrenfeucht game of $k-i$ moves at this point.
\end{itemize}

\vspace{0.1in}

\par We prove this using induction on the number of moves played so far. For $i = 0$, we have chosen $x_{0} = R_{1}, y_{0} = R_{2}$, and we already have imposed the condition $$B(R_{1}; 3^{k+1}) \equiv_{M_{0}; k} B(R_{2}; 3^{k+1})$$ in \eqref{around the root}. So suppose the claim holds for $0 \leq j \leq i-1$. Without loss of generality suppose Spoiler chooses $x_{i} \in T_{1}$. There are two possibilities:
\begin{enumerate}
\item \emph{Inside move:}
\begin{equation} \label{inside move condition}
x_{i} \in \bigcup_{j=0}^{i-1} B(x_{j}; 2 \cdot 3^{k+1-i}).
\end{equation}
So $x_{i} \in B(x_{l}; 2 \cdot 3^{k+1-i})$ for some $0 \leq l \leq i-1$. By the induction hypothesis, 
$$B(x_{l}; 3^{k+1-l}) \equiv_{M_{0}; k} B(y_{l}; 3^{k+1-l}).$$ 
Duplicator now follows his winning strategy of $EHR_{M_{0}}[B(x_{l}; 3^{k+1-l}), B(y_{l}; 3^{k+1-l}); k]$ and picks $y_{i} \in B(y_{l}; 3^{k+1-l})$. This means that, $$d(x_{i}, x_{l}) < 2 \cdot 3^{k+1-i} \quad \Rightarrow \quad B(x_{i}; 3^{k+1-i}) \subset B(x_{l}; 3^{k+1-l}),$$ since $l < i$. In the same way $$B(y_{i}; 3^{k+1-i}) \subset B(y_{l}; 3^{k+1-l}),$$ and further, $$B(x_{i}; 3^{k+1-i}) \equiv_{M_{0}; k} B(y_{i}; 3^{k+1-i}).$$ 
This last relation follows from the fact that $y_{i}$ is chosen corresponding to $x_{i}$ in the winning strategy of the Duplicator for $EHR_{M_{0}}[B(x_{l}; 3^{k+1-l}), B(y_{l}; 3^{k+1-l}); k]$. Since $M_{0}$, as chosen in Equation \eqref{M value}, is greater than $2 \cdot 3^{k+1-i}$, hence for Duplicator to win $EHR_{M_{0}}[B(x_{l}; 3^{k+1-l}), B(y_{l}; 3^{k+1-l}); k]$, he must be able to win the game played within the smaller balls $B(x_{i}; 3^{k+1-i})$ and $B(y_{i}; 3^{k+1-i})$.

\item \emph{Outside move:}
\begin{equation} \label{outside move condition}
x_{i} \notin \bigcup_{j=0}^{i-1} B(x_{j}; 2 \cdot 3^{k+1-i}).
\end{equation}
Then we consider $B(x_{i}; 3^{k+1-i})$ and we know, from \eqref{distance from previous vertices}, \eqref{distance from root} and \eqref{copy of each equivalence class}, that there exists some $v \in T_{2}$ such that $$d(v, y_{l}) > 3^{k+2}, \quad \text{for all } 0 \leq l \leq i-1,$$ and $$B(v; 3^{k+1}) \equiv_{M_{0}; k} B(x_{i}; 3^{k+1}).$$ We choose $y_{i} = v$. Note that then we automatically have $$B(y_{i}; 3^{k+1-i}) \bigcap \left\{\bigcup_{j=0}^{i-1} B(y_{j}; 3^{k+1-i})\right\} = \phi,$$ and $$B(x_{i}; 3^{k+1-i}) \equiv_{M_{0}; k} B(y_{i}; 3^{k+1-i}).$$
Once again, Duplicator is choosing $y_{i}$ so that $B(y_{i}; 3^{k+1}) \equiv_{M_{0}; k} B(x_{i}; 3^{k+1})$, i.e. he wins $$EHR_{M_{0}}\left[B(x_{i}; 3^{k+1}), B(y_{i}; 3^{k+1}); k\right].$$ Then he must be able to win the game within the smaller balls $B(x_{i}; 3^{k+1-i})$ and $B(y_{i}; 3^{k+1-i})$, since his winning involves being able to preserve mutual distances of pairs of nodes up to $M_{0}$.

\end{enumerate}

\par This shows that the Duplicator will win $EHR[T_{1}, T_{2}; k]$, which finishes the proof.

\end{proof}

\begin{thm}\label{xmastree}
For each $k\in \mathbb{N}$ there is a universal tree $T$.
\end{thm}

\begin{proof}
\par $T$ will be a \emph{Christmas tree} which is constructed as follows.  For each $\sig\in \Sig_{M_0:k}$ select and fix a specific
ball $B(v;3^{k+1})\in \sig$.  For each such $\sig$ and each $1\leq i\leq k$ create
disjoint copies $T_{i,\sig} = B(v_{i,\sig};3^{k+1})$ such that $B(v_{i; \sig}; 3^{k+1}) \cong B(v; 3^{k+1})$, with the isomorphism mapping $v_{i; \sig}$ to $v$. These $B(v_{i; \sig}; 3^{k+1})$ are the \emph{balls} decorating the Christmas tree. Let $w_{i; \sig}$ be the \emph{top} vertex of $B(v_{i;\sig};3^{k+1})$. That is, it is that unique node in the ball with no ancestor in the ball. It can be seen that this node is actually the ancestor of $v_{i; \sig}$ which is at distance $3^{k+1}$ away from $v_{i; \sig}$, or in other words, $v_{i; \sig}$ is a $3^{k+1}$-descendant of this node. Let $R$ be the root of $T$. Draw disjoint paths of length $3^{k+4}$ from $R$ to each $w_{i; \sig}$. These will be like the \emph{strings} attaching the balls to the Christmas tree. 

\par We now explain why this $T$ satisfies Conditions (\ref{point 1}) and (\ref{point 2}). Once again, for pedagogical clarity, we first show a detailed reasoning why $T$ satisfies (\ref{point 1}), although technically, it suffices to verify only (\ref{point 2}). First, observe that the $v_{i; \sigma}$ we have defined in the previous paragraph, for $1 \leq i \leq k$ and $\sigma \in \Sigma_{M_{0};k}$, immediately satisfy \eqref{distance from root of universal tree} and \eqref{distance from each other}, since $$d\left(R, v_{i; \sigma}\right) = d\left(R, w_{i; \sigma}\right) + d\left(v_{i; \sigma}, w_{i; \sigma}\right) = 3^{k+4} + 3^{k+1} > 3^{k+2},$$
for every $\sigma_{1}, \sigma_{2} \in \Sigma_{M_{0};k}, 1 \leq i_{1}, i_{2} \leq k$ with $(\sigma_{1}, i_{1}) \neq (\sigma_{2}, i_{2})$, we indeed have
$$d\left(v_{i_{1};\sigma_{1}}, v_{i_{2};\sigma_{2}}\right) = d\left(v_{i_{1};\sigma_{1}}, R\right) + d\left(R, v_{i_{2};\sigma_{2}}\right) > 2\cdot 3^{k+4} > 3^{k+2}.$$
To see that \eqref{each equivalence class represented} holds, note that by our construction, $$B(v_{i; \sig}; 3^{k+1}) \cong B(v; 3^{k+1}) \in \sigma,$$ with $v_{i;\sigma}$ mapped to $v$, for all $1 \leq i \leq k$, and for all $\sigma \in \Sig_{M_0:k}$. 
\par Finally, we verify that (\ref{point 2}) holds. Consider any $1 \leq j \leq k$. Suppose we have selected any $j-1$ vertices $u_{1}, \ldots u_{j-1}$ from $T$. For any $\sigma \in \Sig_{M_0:k}$ and $1 \leq i \leq k$, we consider the \emph{branch} of the tree consisting of the ball $B\left(v_{i;\sigma}; 3^{k+1}\right)$ and the string joining $R$ to $w_{i;\sigma}$, and we call that branch \emph{free} if no $u_{l}, 1 \leq l \leq j-1$ is picked from that branch. Since there are $k$ copies of balls for each $\sigma$, and $j \leq k$, hence we shall always have at least one \emph{free} branch from each $\sigma \in \Sig_{M_0:k}$. So we simply choose $u_{j} = v_{i;\sigma}$ for some $i$ such that the corresponding branch is free. 

\par Since no $u_{l}, \ 1 \leq l \leq j-1$, belongs to that branch, each of them must be at least as far away from $u_{j}$ as the root is from $v_{i;\sigma}$. That is, we will have $$d\left(u_{j}, u_{l}\right) > 3^{k+4}+3^{k+1}; \quad d\left(u_{j}, R\right) = 3^{k+4}+3^{k+1}.$$ And of course, by our choice, we would have $B\left(u_{j}; 3^{k+1}\right) \in \sigma$.\\

\end{proof}

\section{Probabilities conditioned on infiniteness of the tree}\label{sectprob}
As before,  with $R$ the root, $B_{T}(R; i)$ denotes 
the neighbourhood of $R$ with radius $i$, i.e. $$B_{T}(R; i) = \{u \in T: d(u, R) <  i\}.$$ We define 
$$\overline{B_{T}(R; i)} =\{u \in T: d(u, R) \leq  i\}.$$ So, $\overline{B_{T}(R; i)}$ captures up
to the $i$-th generation of the tree, $R$ being the $0$-th generation. For each $i\in \mathbb{N}$ we give a set of
equivalence classes $\Gam_i$ which will be relatively easy to handle and which we show in Theorem \ref{refine}
is a refinement of $\Sig_{i:k}$. 
We set
\begin{equation}\label{setc}
C = \{0,1,\ldots,k-1,\omega\}.
\end{equation}
Here $\omega$ is a special symbol with the meaning ``at least $k$."  That is, to say that there are $\omega$ copies
of someting is to say that there are at least $k$ copies.  We set
\begin{equation}\label{gamma1}
\Gam_1 = C = \{0,1,\ldots,k-1,\omega\}.
\end{equation}
A $\overline{B_T(R;1)}$ is of type $i\in \Gam_1$ if the root has $i$ children.  Since the game has $k$ rounds, if the roots
has $x,y$ children in the two trees with both $x,y\geq k$ then Duplicator wins the modified game. 
Inductively
we now set
\begin{equation}\label{gammai}
\Gam_{i+1} = \{g: \Gam_i\ra C \}.
\end{equation}
Each child $v$ of the root generates a tree to generation $i$.  This tree belongs to an equivalence class
$\sig\in \Gam_i$.  A $\overline{B_T(R;i+1)}$ has state $g\in \Gam_{i+1}$ if for all $\sig\in \Gam_i$ the root 
has $g(\sig)$ children $v$ whose subtree $T(v)$ upto generation $i$ belongs to equivalence class $\sig$, i.e. $T(v)|_{i} \in \sig$.
\\ 
\begin{example} \label{example of neighbourhood}
Consider $k=4,i=2$. Then a typical example of $\overline{B_{T}(R; i)}$ will be: the root has two children with no chidren, at least four children with one child, three children with two children, no children with three children, and one child with
at least four children. Thus $g(0) = 2, g(1) = \omega, g(2) = 3, g(3) = 0, g(\omega) = 1$.
\end{example}

\begin{thm}\label{refine} $\Gam_i$ is a refinement on $\Sig_{i:k}$.  \end{thm}

\begin{proof} Let $\overline{B_{T_1}(R_1;i)}, \overline{B_{T_2}(R_2,i)}$ lie in the same $\Gam_i$ equivalence class.
It suffices to show that Duplicator wins the $k$-move modified Ehrenfeucht game on these balls. We show this using induction on $i$. 
\par The case $i = 1$ is immediate. Suppose it holds good for all $i' \leq i-1$. In the Ehrenfeucht game let Spoiler select $w_1\in T_1$.  Let $v_1$ be the child of the root such that $w_1$ belongs to the tree generated by $v_1$ up to depth $i-1$, i.e. $T_{1}(v_{1})|_{i-1}$. Duplicator allows Spoiler a free move of $v_1$. Let $\sig$ be the $\Gam_{i-1}$ class for $T_{1}(v_{1})|_{i-1}$.  In $T_2$ Duplicator finds a child $v_2$ of the root $R_{2}$ in $T_{2}$ such that $T_{2}(v_{2})|_{i-1} \in \sig$. Duplicator now moves $v_2$ and then, by induction hypothesis, finds the appropriate response $w_2 \in T_{2}(v_{2})|_{i-1}$ corresponding to $w_{1}$.  For any further moves by the Spoiler with the same $v_1$ or $v_2$, Duplicator plays, inductively, on the two subtrees $T_{1}(v_{1})|_{i-1}, T_{2}(v_{2})|_{i-1}$.  And if Spoiler chooses some $y_{1} \in B_{T_{1}}(R_{1};i) - T_{1}(v_{1})|_{i-1}$, then again we repeat the same procedure as above. There
are only $k$ moves, hence Duplicator can continue in this manner and so wins the Ehrenfeucht game.

\end{proof}

When $\sig\in \Gam_i$ we write $\Pr[\sig],\Pr^*[\sig]$ for the probabilities, in $T_{\lam},
T_{\lam}^*$ respectively, that $\overline{B_T(R,i)}$ is in equivalence class $\sig$.
Let $\Gam=\Gam_s$ with $s=3^{k+1}$.

For any first order $A$ with quantifier depth $k$ let $J_A$ be as in (\ref{JA}).  Applying Theorem \ref{refine}
for each $i\in J_A$ the class $\mathcal{B}_{i}$ splits into finitely many classes $\tau\in \Gam$.  Let
$K_A$ denote the set of such classes.  The equation (\ref{anyA}) can be rewritten as

\begin{equation}\label{anyAgamma} \Pr^*[A] = \sum_{\tau\in K_{A}} \Pr^*[\tau]. \end{equation}

For $0\leq i < k$ set
\begin{equation}\label{Poi}
P_i(x)= \Pr[Po(x)=i ] = e^{-x}\frac{x^i}{i!},
\end{equation}
and set
\begin{equation}\label{Poomega}
P_{\omega}(x)= \Pr[Po(x)\geq k ] =  1 - \sum_{i=0}^{k-1}P_{i}(x).
\end{equation}

We now make use of a special property of the $Poisson$ distribution.
Let $\Omega=\{1,\ldots,n\}$ be some finite state space.  Let
$p_i\geq 0$ with $\sum_{i=1}^n p_i = 1$ be some distribution over $\Omega$.
Suppose $v$ has $Poisson$ mean $\lam$ children and each child 
independently is in state $i$ with probability $p_i$.  The distribution
of the number of children of each type is the same as if for each
$i\in \Omega$ there were $Poisson$ mean $p_i\lam$ children of type $i$
and these values were mutually independent.  For example, assumming
boys and girls equally probable, having $Poisson$ mean $5$ children is
the same as having $Poisson$ mean $2.5$ boys and, independently, having
$Poisson$ mean $2.5$ girls.

The probability, in $T_{\lam}$, that the root has $u$ children
(including $u=\omega$) is then $P_u(\lam)$.  Suppose, by induction,
that $P_{\tau}(x)$ has been defined for all $\tau\in \Gam_i$ such
that $\Pr(\tau)=P_{\tau}(\lam)$.  Let $\sig\in \Gam_{i+1}$ so
that $\sig$ is a function $g: \Gam_i\ra C$.  In $T_{\lam}$ the
root has $Poisson$ mean $\lam$ children and, for each $\tau\in \Gam_i$,
the $i$-generation tree rooted at a child is in the class $\tau$ with
probability $P_{\tau}(\lam)$.  By the special property above we
equivalently say that the root has $Poisson$ mean $\lam P_{\tau}(\lam)$
children of type $\tau$ for each $\tau\in \Gam_i$ and that these
numbers are \emph{mutually independent}.  The probability 
$P_{\sig}(\lam)$ is then the product, over $\tau\in \Gam_i$, of the
probability that a $Poisson$ mean $\lam P_{\tau}(\lam)$ has value
$g(\tau)$.  Setting
\begin{equation}\label{recur}
     P_{\sig}(x) = \prod_{\tau} P_{g(\tau)}(x P_{\tau}(x)),
\end{equation}
we have 
\begin{equation}\label{recur1}
\Pr[\sig] = P_{\sig}(\lam). 
\end{equation}

\begin{example} Continuing Example \ref{example of neighbourhood}, set
$x_i= e^{-\lam}\lam^i/i!$ for $0\leq i < 4$ and $x_{\omega} =
1-\sum_{i=0}^3 x_i$.  The root has no child with three children
with probability $\exp[-x_3\lam]$.  It has one child with at least
four children with probability $\exp[-x_{\omega}\lam](x_{\omega}\lam)$.  It has at
least four children with one child with probability
$1-\exp[-x_1\lam](1+ (x_1\lam) + (x_1\lam)^2/2 + (x_1\lam)^3/6]$.  
It has two children with no children with probability
$\exp[-x_0\lam](x_0\lam)^2/2$.
It has three children with two children with probability
$\exp[-x_2\lam](x_2\lam)^3/6$.
The probability of the
event is then the product of these five values.
\end{example}

While Equation (\ref{recur1}) gives a very full
description of the possible $\Pr[\sig]$ the following less precise
description may be more comprehensible.  

\begin{defn}\label{defnice} Let $\mathcal{F}$ be
the minimal family of function $f(\lam)$ such that
\ben
\item $\mathcal{F}$ contains the identity function $f(\lam)=\lam$ and the constant functions $f_{q}(\lam) = q, q \in \mathbb{Q}$.
\item $\mathcal{F}$ is closed under finite addition, subtraction and multiplication.
\item $\mathcal{F}$ is closed under base $e$ exponentiation.  That
is, if $f(\lam)\in \mathcal{F}$ then $e^{f(\lam)}\in \mathcal{F}$.
\een
We call a function $f(\lam)$ {\em nice} if it belongs to $\mathcal{F}$.
\end{defn}
In Corollary \ref{a54} we show that the probability of any first order property, conditioned on the tree being infinite, is actually such a nice function.

\begin{thm}\label{thmnice}
Then for all $k$ and all $i$, if $\sig\in \Gam_i$ then
$\Pr[\sig]$ is a nice function of $\lam$.
\end{thm}

This is an immediate consequence of the recursion (\ref{recur}).

\begin{example} The statement ``the root has no children which
have no children which have no children" is the union of classes
$\sig$ with $k=1$, $i=3$.  It has probability
$\exp[-\lambda \exp[-\lambda \exp[-\lambda]]].$ 
\end{example} 

Let $T_{\lam}^{fin}$ denote
$T_{\lam}$ conditioned on $T_{\lam}$ being finite. For any $k,i$ and any
$\sig\in \Gam_i$ let $\Pr^{fin}[\sig]$ be the probability of event $\sig$
in $T^{fin}$.  Assume $\lam > 1$.
Let $p=p(\lam)$, the probability $T_{\lam}$ is infinite, be given by (\ref{defp}).
By duality, $T_{\lam}^{fin}$ has the same distribution as $T_{q\lam}$, where 
\begin{equation} \label{prob of inf}
q(\lambda) = 1 - p(\lambda) = \Pr[T_{\lambda} \text{ is finite}].
\end{equation}
Thus

\begin{equation}\label{recurfinite}
\Pr^{fin}[\sig] = P_{\sig}(q\lam).
\end{equation}

For any $k,i$ and $\sig\in \Gam_i$ 
\begin{equation}\label{a50}
\Pr[\sig] = \Pr^{fin}[\sig]q + \Pr^*[\sig]p  
\end{equation}
and hence
\begin{equation}\label{a51}
\Pr^*[\sig] = p^{-1}[\Pr[\sig] - \Pr^{fin}[\sig]q].
\end{equation}
For any first order sentence $A$ of quantifier depth $k$, letting $K_A$
be as in (\ref{anyAgamma}),

\begin{equation}\label{a60}
\Pr^*[A] = \sum_{\sig\in K_A} 
 p^{-1}[\Pr[\sig] - \Pr^{fin}[\sig]q].
\end{equation}

Combining previous results gives a description of possible $\Pr^*[A]$.

\begin{thm}\label{a52}
Let $A$ be a first order sentence of quantifier depth $k$.  Let
$K_A$ be as in (\ref{anyAgamma})  Let
\begin{equation}\label{setf}
f(x)=\sum_{\sig\in K_A} P_{\sig}(x).
\end{equation}
Then 
\begin{equation}\label{a53}
\Pr^*[A] = p^{-1}[f(\lam)-qf(q\lam)].
\end{equation}
\end{thm}

As before, it is also convenient to give a slightly weaker form.

\begin{cor}\label{a54}  For any first order sentence $A$ we may
express
\begin{equation}\label{a55}
\Pr^*[A] = p^{-1}[f(\lam)-qf(q\lam)]
\end{equation}
where $f$ is a nice function in the sense of Definition \ref{defnice}.
\end{cor}

\section{Further results}
In this paper, we have so far dealt with Galton-Watson trees with $Poisson$ offspring distribution. The results of Sections \ref{sectinf} and \ref{sectionuniv} extend to some other classes of offspring distributions. In this section, we outline briefly these extensions. We consider a general probability distribution $D$ on $\mathbb{N}_{0} = \{0, 1, 2, \ldots\}$, where $p_{i}$ is the probability that a typical node in the random tree has exactly $i$ children, $i \in \mathbb{N}_{0}$. We shall denote the probabilities under this regime by $\Pr_{D}$. We also assume that the moment generating function of $D$ exists on a non-degenerate interval $[0, \gamma]$ on the real line.
\par Fix an arbitrary finite $T_{0}$ of depth $d_{0}$. We assume that $\Pr_{D}[T_{0}] > 0$. In other words, this means that if $T$ is the random Galton-Watson tree with offspring distribution $D$, then $\Pr_{D}[T \cong T_{0}] > 0$. Consider the statement
\begin{equation}
A = \left\{\exists \ v : T(v) \cong T_{0}\right\} \vee \left\{T \text{ is finite}\right\}.
\end{equation}
We can show, similar to our results in Section \ref{sectinf}, that $\Pr_{D}[A] = 1$, provided \eqref{general distribution condition} holds for some $\alpha \in (0, \gamma]$ and $0 < \epsilon < 1$. Of course, the non-trivial case to consider is when $D$ has expectation greater than $1$, as only then does it make sense to talk about the infinite Galton-Watson tree.
\par The proof of this fact follows the exact same steps as shown in Section \ref{sectinf}. We consider again a fictitious continuation $X_{1}, X_{2}, \ldots$ which are i.i.d.\ $D$. For every node $i$, we let $I_{i}$ be the indicator for the event $\{T(i) \cong T_{0}\}$. For a suitable $\epsilon > 0$ that we choose later, we let 
\begin{equation}
Y = \sum_{i=1}^{\left\lfloor \epsilon^{d_{0}} s\right\rfloor} I_{i},
\end{equation}
and we define the martingale $Y_{i} = E[Y|X_{1}, \ldots, X_{i}]$ for $1 \leq i \leq s$, with $Y_{0} = E[Y]$. Defining $g_{1}$ as in Equation \eqref{a23}, we similarly argue that 
\begin{equation}
g_{1}(x) \leq \lfloor x \rfloor + \sum_{i=1}^{\lfloor x \rfloor} X_{i}.
\end{equation}
The only difference is in the estimation of the probability that $g_{1}(\epsilon x)$ exceeds $x$. We employ Chernoff bounds again, but we no longer have the succinct form of the moment generating function as in the case of $Poisson$. For any $0 < \alpha \leq \gamma$, 
\begin{align} \label{general distribution Chernoff}
\Pr[g_{1}(\epsilon x) > x] =& \Pr[e^{\alpha g_{1}(\epsilon x)} > e^{\alpha x}] \nonumber\\
\leq & E[e^{\alpha g_{1}(\epsilon x)}] e^{-\alpha x} \nonumber\\
\leq & E[e^{\alpha (\epsilon x + \sum_{i=1}^{\lfloor \epsilon x \rfloor} X_{i})}] e^{-\alpha x} \nonumber\\
=& e^{\alpha \epsilon x} \prod_{i=1}^{\lfloor \epsilon x \rfloor} E[e^{\alpha X_{i}}] e^{-\alpha x} \nonumber\\
=& \varphi(\alpha)^{\lfloor \epsilon x \rfloor} e^{-\alpha (1 - \epsilon) x},
\end{align}
where $\varphi(\alpha) = E[e^{\alpha X_{1}}]$. Since $X_{1}$ is non-negative valued, $\varphi(\alpha) > 1$ for $\alpha > 0$, hence we can bound the expression in \eqref{general distribution Chernoff} above by 
\begin{equation}
\varphi(\alpha)^{\epsilon x} e^{-\alpha (1 - \epsilon)x} = \left\{\varphi(\alpha)^{\epsilon} e^{-\alpha (1 - \epsilon)}\right\}^{x}.
\end{equation}
If we are able to choose $\alpha > 0$ such that for some $0< \epsilon < 1$, we have 
\begin{equation} \label{general distribution condition}
\varphi(\alpha)^{\epsilon} e^{-\alpha (1 - \epsilon)} < 1,
\end{equation}
then the exact same argument as in Section \ref{sectinf} goes through, and we have the desired result.
\par In particular, it is easy to see that \eqref{general distribution condition} is indeed satisfied when $D$ is a probability distribution on a finite state space $\subseteq \mathbb{N}_{0}$.
\par The sufficient conditions for a tree to be universal nowhere uses the offspring distribution. Once the results of Section \ref{sectinf} hold for a given $D$, it is not too difficult to see that the conclusion of Remark \ref{how neighbourhood of root suffices} should hold in this regime as well. We hope to return to this more general setting in our future work.
\par A further object of future study is a more detailed analysis of $T_{\lambda}$ at the critical value $\lambda = 1$. While $\Pr^{*}$ is technically not defined at the critical value, there may well be some approaches via the insipient infinite tree.

\bibliography{mybibfile}
\nocite{*}

\end{document}